\documentclass[12pt]{article}
\usepackage{amssymb,epsf,graphicx,bm,bbm,chngcntr}
\def\msy{\mathbb}

\def\bbbr{{\msy R}}
\def\bbbc{{\msy C}}
\def\bbbq{{\msy Q}}
\def\bbbn{{\msy N}}
\def\bbbp{{\msy P}}
\def\bbbz{{\msy Z}}
\def\bbbf{{\msy F}}

\def\bbbg{{\msy G}}

\def\boldalpha{\bm{\alpha}}
\def\boldbeta{\bm{\beta}}
\def\v#1{{\bf#1}}
\def\is{\equiv}

\def\q{{\mathbbm{q}}}
\def\p{{\mathfrak{p}}}
\def\P{{\mathfrak{P}}}
\def\mod#1{\ ({\rm mod}\ #1)}
\def\fref#1{{\rm(\ref{#1})}}

\def\fqstar{\bbbf_q^{\times}}
\newenvironment{proof}[1]{\noindent {\bf Proof #1: }}{\hfill\break\medskip\hfill$\Box$\medskip}

\newtheorem{theorem}[subsection]{Theorem}
\newtheorem{definition}[subsection]{Definition}
\newtheorem{lemma}[subsection]{Lemma}

\newtheorem{corollary}[subsection]{Corollary}

\newtheorem{proposition}[subsection]{Proposition}
\newtheorem{assumption}[subsection]{Assumption}
\counterwithin{equation}{section}
\def\v#1{{\bf #1}}

\title{Finite hypergeometric functions}
\author{Frits Beukers, Henri Cohen, Anton Mellit}
\date{September 20}

\begin{document}
\maketitle
\medskip

\centerline{\it On the occasion of Eduard Looijenga's 69-th birthday.}
\medskip

\abstract{We show that values of finite hypergeometric functions defined over
$\bbbq$ correspond to point counting results on explicit varieties defined
over finite fields.}

\parindent=0pt

\section{Introduction}
Finite hypergeometric functions were introduced independently, and with
different motivations, by John Greene \cite{greene} and Nick Katz \cite[p.258]{katz}
by the end of the 1980's.
They are complex-valued functions on finite fields and an analogue of the classical
analytic hypergeometric functions in one variable, also known as Thomae's functions.
To display the analogy, we first recall the definition of the analytic hypergeometric functions.

Consider two multisets
(sets with possibly repeating elements) $\boldalpha=(\alpha_1,\ldots,\alpha_d)$
and $\boldbeta=(\beta_1,\ldots,\beta_d)$, where $\alpha_i,\beta_j\in\bbbq$ for all
$i,j$. We assume for the moment that none of the $\alpha_i,\beta_j$ is in $\bbbz_{\le0}$
and $\beta_d=1$.
The hypergeometric function with parameter sets $\boldalpha,\boldbeta$ is
the analytic function defined by the power series in $z$
$$_dF_{d-1}(\boldalpha,\boldbeta|z)=\sum_{n=0}^{\infty}
{(\alpha_1)_n\cdots(\alpha_d)_n\over(\beta_1)_n\cdots(\beta_d)_n}\ z^n\;.$$
Here $(x)_n=x(x+1)\cdots(x+n-1)$ denotes the so-called Pochhammer symbol
or rising factorial. Using the $\Gamma$-function we can rewrite the series as
$${\Gamma(\beta_1)\cdots\Gamma(\beta_d)\over\Gamma(\alpha_1)\cdots\Gamma(\alpha_d)}
\sum_{n=0}^{\infty}{\Gamma(\alpha_1+n)\cdots\Gamma(\alpha_d+n)\over\Gamma(\beta_1+n)
\cdots\Gamma(\beta_d+n)}\ z^n\;.$$
In order to display the similarity with the upcoming definition of finite hypergeometric
functions, using the identity $\Gamma(x)\Gamma(1-x)=\pi/\sin(\pi x)$ we rewrite this as
$$\sum_{n=0}^{\infty}\prod_{i=1}^d\left({\Gamma(\alpha_i+n)(1-\beta_i-n)
\over\Gamma(\alpha_i)\Gamma(1-\beta_i)}\right)\ (-1)^{dn}z^n\;.$$
If $\beta$ is a positive integer we interpret $\Gamma(1-\beta-n)/\Gamma(1-\beta)$ as
$(-1)^n/(n+\beta-1)!$, the limit of $\Gamma(1-x-n)/\Gamma(1-x)$ as $x\to\beta$.

We now define finite hypergeometric functions. Again take two multisets $\boldalpha$
and $\boldbeta$, each consisting of $d$ elements in $\bbbq$. Assume from now on that
$\alpha_i\not\is\beta_j\mod{\bbbz}$ for all $i,j$. In the analytic case this condition
is equivalent to the irreducibility of the hypergeometric differential equation. In the finite case
we avoid certain degeneracies with this assumption. We need not assume $\beta_d=1$ any more.

Let $\bbbf_q$ be the finite field with $q$ elements.
Let $\psi_q$ be a non-trivial additive character on $\bbbf_q$
which we fix once and for all throughout this paper. For any multiplicative character
$\chi:\fqstar\to\bbbc^{\times}$ we define the Gauss sum
$$g(\chi)=\sum_{x\in\fqstar}\chi(x)\psi_q(x)\;.$$
Let $\omega$ be a generator of the character group on $\fqstar$
which we also fix throughout the paper. We use the notation
$g(m)=g(\omega^m)$ for any $m\in\bbbz$. Note that $g(m)$ is periodic in $m$ with
period $q-1$.
Very often we shall need characters on $\fqstar$ of a given order.
For that we use the notation $\q=q-1$ so that a character of order $d$ can be given by
$\omega^{\q/d}$ for example, provided that $d$ divides $\q$ of course.

Now we define finite hypergeometric sums. Let again $\boldalpha$ and $\boldbeta$ be multisets
of $d$ rational numbers each, and disjoint modulo $\bbbz$.
Suppose in addition that $q$ is such that
$$(q-1)\alpha_i,(q-1)\beta_j\in\bbbz$$
for all $i$ and $j$.

\begin{definition}[Finite hypergeometric sum]\label{definitionHq}
Keep the above notation. We define for any $t\in\bbbf_q$,
$$H_q(\boldalpha,\boldbeta|t)={1\over 1-q}\sum_{m=0}^{q-2}
\prod_{i=1}^d\left({g(m+\alpha_i\q)g(-m-\beta_i\q)
\over g(\alpha_i\q)g(-\beta_i\q)}\right)\ \omega((-1)^dt)^m\;.$$
\end{definition}

Note the analogy with the analytic hypergeometric function.
These sums were considered without the normalizing factor $(\prod_{i=1}^dg(\alpha_i\q)g(-\beta_i\q))^{-1}$
by Katz in \cite[p258]{katz}. Greene, in \cite{greene}, has a definition involving Jacobi sums which, after some
elaboration, amounts to
$$\omega(-1)^{|\boldbeta|\q}q^{-d}\prod_{i=1}^d{g(\alpha_i\q)g(-\beta_i\q)\over g(\alpha_i\q-\beta_i\q)}
\ H_q(\boldalpha,\boldbeta|t)\;,$$
where $|\boldbeta|=\beta_1+\cdots+\beta_d$.
The normalization we adopt in this paper coincides with that of McCarthy, \cite[Def 3.2]{mccarthy}.

It took some time before people realized that Greene's finite hypergeometric functions were closely related
to point counting results on algebraic varieties over finite fields. As an example we mention the following
result, adapted to our notation.

\begin{theorem}[K.~Ono, 1998]\label{onotheorem}
Let $q$ be an odd prime power and $\lambda\in\bbbf_q$ and $\lambda\ne0,1$. Let $E_{\lambda}$ be the projective
elliptic curve given by the affine equation $y^2=x(x-1)(x-\lambda)$ and $E_{\lambda}(\bbbf_q)$
the set of $\bbbf_q$-rational points (including infinity). Then
$$|E_{\lambda}(\bbbf_q)|=q+1-(-1)^{(q-1)/2}H_q(1/2,1/2;1,1|\lambda)\;.$$
\end{theorem}

In this paper we propose a generalization of this theorem which applies to all hypergeometric sums which
are {\it defined over $\bbbq$}. By that we mean that both $\prod_{j=1}^d(x-e^{2\pi i\alpha_j})$ and
$\prod_{j=1}^d(x-e^{2\pi i\beta_j})$ are polynomials with coefficients in $\bbbz$. In other words, they
are products of cyclotomic polynomials. Let us assume we are in this case. Then we can find natural numbers
$p_1,\ldots,p_r$ and $q_1,\ldots,q_s$ such that
$$\prod_{j=1}^d{x-e^{2\pi i\alpha_j}\over x-e^{2\pi i\beta_j}}={\prod_{j=1}^r x^{p_j}-1\over
\prod_{j=1}^s x^{q_j}-1}\;.$$
Note that $p_1+\cdots+p_r=q_1+\cdots+q_s$. It is a small exercise to show that
$${(\alpha_1)_n\cdots(\alpha_d)_n\over(\beta_1)_n\cdots(\beta_d)_n}
=M^{-n}{(p_1n)!\cdots(p_rn)!\over(q_1n)!\cdots(q_sn)!},
\quad M={p_1^{p_1}\cdots p_r^{p_r}\over q_1^{q_1}\cdots q_s^{q_s}}\;.$$

It turns out that when the hypergeometric parameters are defined over $\bbbq$, it is possible to
extend the definition of $H_q$ to all prime powers $q$ which are relatively prime to the common
denominator of the $\alpha_i,\beta_j$.
Let $D(X)$ be the greatest common divisor of
the polynomials $\prod_{i=1}^r(X^{p_i}-1)$ and $\prod_{j=1}^s(X^{q_j}-1)$.
We rewrite our Katz sum in a different shape.

\begin{theorem}\label{rewriteZ}
With the above notation we have
$$H_q(\boldalpha,\boldbeta|t)={(-1)^{r+s}\over 1-q}
\sum_{m=0}^{q-2}q^{-s(0)+s(m)}g(\v pm,-\v qm)
\ \omega(\epsilon M^{-1}t)^m\;,$$
where
$$g(\v pm,-\v qm)=g(p_1m)\cdots g(p_rm)g(-q_1m)\cdots g(-q_sm),
\quad M=\prod_{j=1}^rp_j^{p_j}\prod_{j=1}^s q_j^{-q_j}$$
and $\epsilon=(-1)^{\sum_iq_i}$
and $s(m)$ is the multiplicity of the zero $e^{2\pi im/\q}$ in $D(X)$.
\end{theorem}

This theorem is proven in Section \ref{overQ}.

\begin{assumption}
From now on, when we work over $\bbbq$,
we adopt the right-hand side of Theorem \ref{rewriteZ} as definition of $H_q$.
\end{assumption}

Let $\lambda\in\fqstar$ and let $V_{\lambda}$ be the affine variety defined by the projective equations
\begin{equation}\label{Vequation}
x_1+x_2+\cdots+x_r-y_1-\cdots-y_s=0,\qquad \lambda x_1^{p_1}\cdots x_r^{p_r}=y_1^{q_1}\cdots y_s^{q_s}
\end{equation}
and $x_j,y_j\ne0$. The main theorem of this paper reads as follows.

\begin{theorem}\label{main}
Let the notation for $p_i,q_j,M$ be as above. Suppose that the greatest common divisor of
$p_1,\ldots,p_r,q_1,\ldots,q_s$ is one and suppose that $M\lambda\ne 1$. Then there
exists a suitable non-singular completion of $V_{\lambda}$, denoted by $\overline{V_{\lambda}}$, such that
$$|\overline{V_{\lambda}}(\bbbf_q)|=P_{rs}(q)+(-1)^{r+s-1}q^{\min(r-1,s-1)}H_q(\boldalpha,\boldbeta|M\lambda)\;,$$
where
$$P_{rs}(q)=\sum_{m=0}^{\min(r-1,s-1)}{r-1\choose m}{s-1\choose m}{q^{r+s-m-2}-q^m\over q-1}$$
and $\overline{V_{\lambda}}(\bbbf_q)$ is the set of $\bbbf_q$-rational points on $\overline{V_{\lambda}}$.
\end{theorem}
What we mean by a suitable non-singular completion of $\overline{V_{\lambda}}$ is elaborated in
Section \ref{completion}, more precisely Definition \ref{completiondefinition}. Furthermore,
one easily checks that the condition $M\lambda\ne0,1$ corresponds to the non-singularity
of $V_{\lambda}$.

Here we give a few examples to illustrate the theorem.
\begin{corollary}Let $f(x)=x^3+3x^2-4t$ with $t\in\bbbf_q$ and $t\ne0,1$. Let $N_f(t)$ be the
number of zeros of $f(x)$ in $\bbbf_q$. Suppose that $q$ is not divisible by $2$ or $3$. Then
$$N_f(t)=1+H_q(1/3,2/3;1,1/2|t)\;.$$
\end{corollary}
\begin{proof}{} We take $\alpha_1=1/3,\alpha_2=2/3,\beta_1=1,\beta_2=1/2$ and note that
$${(x-e^{2\pi i/3})(x-e^{4\pi i/3})\over(x-1)(x+1)}={x^3-1\over(x-1)(x^2-1)}\;.$$
So $p_1=3,\ q_1=1,\ q_2=2$, and $M=27/4$. The variety $V_{\lambda}$ is given by the equations
$x-y_1-y_2=0,\lambda x^3=y_1y_2^2$. Eliminate $y_1$ and set $y_2=1$ (dehomogenization) to get
$x-\lambda x^3-1=0$. Replace $x$ by $-3/x$ to get $x^3+3x^2-27\lambda=0$. Application of
Theorem \ref{main} and the relation $t=27\lambda/4$ gives the desired result.
\end{proof}

\begin{corollary}Consider the elliptic curve $y^2+xy+y=\lambda x^3$ with $\lambda\in\bbbf_q$.
Denote its completion with the point at infinity by $E_{\lambda}$. Suppose that $q$ is not divisible
by $2$ or $3$. Then
$$|E_{\lambda}(\bbbf_q)|=q+1-H_q(1/3,2/3;1,1|27\lambda)\;.$$
\end{corollary}

\begin{proof}{}
We take $\alpha_1=1/3,\alpha_2=2/3,\beta_1=1,\beta_2=1$ and note that
$${(x-e^{2\pi i/3})(x-e^{4\pi i/3})\over(x-1)^3)}={(x^3-1)\over(x-1)^3}\;.$$
So $p_1=3,\ q_1=q_2=q_3=1$, and $M=27$. The variety $V_{\lambda}$ is given by the equations
$$x_1-y_1-y_2-y_3=0,\lambda x_1^3=y_1y_2y_3\;.$$
Eliminate $y_3$ and set $x_1=1$ to dehomogenize. We get
$$1-y_1-y_2-\lambda y_1^{-1}y_2^{-1}=0\;.$$
Introduce new coordinates $x,y$ via
$$y_1=-y/x,\qquad y_2=-1/x\;.$$
Note that this is a birational map. In the new coordinates get the curve
$$y^2+xy+y=\lambda x^3\;.$$
Theorem \ref{main} tells us that the number of $\bbbf_q$-rational points (including $\infty$) equals
$$q+1-H_q(1/3,2/3;1,1|27\lambda)\;.$$
\end{proof}

\begin{corollary} Consider the rational elliptic surface $S_{\lambda}$
given by the affine equation
$$y^2-xyz+x^3+z^5-\lambda^{-1}=0$$
with $\lambda\in\fqstar$. Suppose that $q$ is not divisible by $2,3$ or $5$ and
$2^{14}3^95^5\lambda\ne 1$. Let $\overline{S_{\lambda}}$ be a suitable completion of $S_{\lambda}$.
Then
$$|\overline{S_{\lambda}}(\bbbf_q)|=q^2+3q+1+qH_q(\boldalpha,\boldbeta|2^{14}3^95^5\lambda)\;,$$
where
$$\boldalpha=(1/30,7/30,11/30,13/30,17/30,19/30,23/30,29/30)$$ and
$$\boldbeta=(1/5,1/3,2/5,1/2,3/5,2/3,4/5,1)\;.$$
Moreover, $H_q(\boldalpha,\boldbeta|t)$ is integer valued.
\end{corollary}

From \cite{beukersheckman} it follows that the analytic function $_8F_7(\boldalpha,\boldbeta|z)$
is an algebraic function. Fernando Rodriguez-Villegas computed its degree over $\bbbc(z)$
\cite{frv}, which is $483840$. He also noted that
$$_8F_7(\boldalpha,\boldbeta| 2^{14}3^95^5t)=\sum_{n\ge0}{(30n)!n!\over(15n)!(10n)!(6n)!}\ t^n\;.$$
The coefficients turn out to be integers and they were essentially used by Chebyshev in his proof
for his estimates of the prime counting function $\pi(x)$.

Another remark is that $y^2-xyz+x^3+z^5-\lambda^{-1}=0$ is a rational elliptic surface for any given
$\lambda$. The $\zeta$-function of such surfaces has been computed extensively by Shioda
in\cite{shioda}. There it turns out that the global $\zeta$-function of such a surface has the form
$\zeta(s)\zeta(s-1)^2\zeta(s-2)L(\rho,s-1)$, where $\zeta(s)$ is the Riemann $\zeta$ function and
$L(\rho,s)$ is the Artin $L$-series corresponding to a finite representation of Gal($\overline{\bbbq}/\bbbq)$)
of dimension $8$.
\medskip

\begin{proof}{}We
take the elements of $\boldalpha$ for the $\alpha_i$ and the elements of $\boldbeta$ for the
$\beta_j$. We verify that
$$\prod_{j=1}^8 {x-e^{2\pi \alpha_j}\over x-e^{2\pi i\beta_j}}=
{(x^{30}-1)(x-1)\over(x^{15}-1)(x^{10}-1)(x^6-1)}\;.$$
Compare the exponents with the remarks above. We have $M=2^{14}3^95^5$. Theorem \ref{main} implies that
$$q^2+3q+1+qH_q(\boldalpha,\boldbeta|M\lambda)$$
equals the number of $\bbbf_q$-rational points on a suitable completion of
$$x_1+x_2-y_1-y_2-y_3=0,\qquad \lambda x_1^{30}x_2=y_1^{15}y_2^{10}y_3^6\;.$$
Eliminate $x_2$ and set $x_1=1$. We obtain
$$1+\lambda^{-1}y_1^{15}y_2^{10}y_3^6-y_1-y_2-y_3=0\;.$$
Substitute $y_1=x^{-1}yz^{-1},y_2=x^2y^{-1}z^{-1},y_3=x^{-1}y^{-1}z^4$. Note that
this monomial substitution is reversible because the matrix of exponents has determinant
$-1$. We obtain
$$-xyz-\lambda^{-1}+y^2+x^3+z^5=0\;.$$
The integrality of the values of $H_q$ follows from Theorem \ref{denominatorHq}.
\end{proof}
\medskip

Surprisingly enough Theorem \ref{main} does not immediately imply Ono's Theorem \ref{onotheorem}.
In this case we get the parameters $p_1=p_2=2,q_1=q_2=q^3=q^4=1$ and the threefold
$$x_1+x_2-y_1-y_2-y_3-y_4=0,\qquad \lambda x_1^2x_2^2=y_1y_2y_3y_4\;,$$
instead of the expected Legendre family of elliptic curves. However in this case we
have the relations $p_1=q_1+q_2,p_2=q_3+q_4$ besides the overall relation $p_1+p_2=q_1+\cdots+q_4$.
In such a case we can construct another variety whose point count also yields the desired hypergeometric
sum. We indicate how this works in Theorem \ref{legendrecount} in the last section of this paper.
This theorem does yield Legendre's family for our example.

{\bf Acknowledgements:} We are greatly indebted to Fernando Rodriguez-Villegas who inspired the subject,
and who co-organized the AIM/ICTP meeting on 'Hypergeometric motives' at Trieste in
July 2012. During this stimulating meeting the idea of reverse engineering hypergeometric
motives on an industrial scale arose. We thank both AIM and ICTP for
their hospitality.
We also like to thank the Mathematisches Forschungsinstitut Oberwolfach
for providing the wonderful atmosphere during the workshop 'Explicit methods in number theory'
in July 2013. Here the final form of the present paper was conceived. 
Finally we thank the anonymous referee for some pertinent questions which led to a
substantial improvement of the manuscript.

\section{Gauss sums, basic properties}

We use the notation given in the introduction and recall some basic theorems.
The reference that we use is the second author's books on Number
Theory, \cite{cohen1}, \cite{cohen2} (where the notation $\tau(\chi,\psi)$ is
used instead of $g(\chi)$ though).
\begin{theorem}\label{cancelgauss}We have
\begin{enumerate}
\item $g(0)=-1$
\item $g(m)g(-m)=\omega(-1)^mq$ if $m\not\is0\mod{\q}$.
\end{enumerate}
\end{theorem}

This is Lemma 2.5.8 and Proposition 2.5.9 of \cite{cohen1}.

\begin{theorem}\label{jacobisum}
Define for any two integers $m,n$ the Jacobi sum
$$J(m,n)={g(m)g(n)\over g(n+m)}\;.$$
Then,
$$J(m,n)=\cases{\sum_{x\in\bbbf_q\setminus\{0,1\}}\omega(1-x)^m\omega(x)^n & \mbox{if $m,n,m+n\not\is0\mod{\q}$}\cr
-1 & \mbox{if $m\is0\mod{\q}$ or $n\is 0\mod{\q} $}\cr
\omega(-1)^mq & \mbox{if $m\is-n\mod{\q}$ and $m\not\is0\mod{\q}$}\cr}$$
\end{theorem}
This follows from Corollary 2.5.17 of \cite{cohen1}, although we slightly
perturbed the definition of $J$.
A simple consequence is the following.

\begin{lemma}\label{integralJ}
For every pair of integers $m,n$ the quotient $g(m)g(n)/g(m+n)$ is an algebraic
integer in $\bbbq(\zeta_{\q})$, where $\zeta_{\q}$ is a primitive $q-1$-st root of
unity.
\end{lemma}

In Theorems \ref{denominatorHqgeneral} and \ref{denominatorHq}
(in the case over $\bbbq$),
we will compute a bound for the denominator of $H_q$. To do so we need a statement on the prime ideal factorization of Gauss sums. This is provided by Stickelberger's theorem.
We use Theorem 3.6.6 of \cite{cohen1}, in a weaker form.

\begin{theorem}[Stickelberger] \label{stickelberger}
Let $\p$ be the ideal divisor of $p$ in $\bbbz[\zeta_{\q}]$ such that $\omega^{-1}(x)\is x\mod{\p}$
for all $x\in\bbbf_q$. Note that $\p$ has degree $f$, where $p^f=q$.
Let $\P$ be the totally ramified prime ideal above in $\p$ in $\bbbz[\zeta_{\q},\zeta_p]$.
Let $0\le r<q-1$. Then
$g(r)$ is exactly divisible by $\P^{\sigma(r)}$, where $\sigma(r)$ is the sum of the digits $r_i$ in the base
$p$ decomposition $r=r_0+r_1p+\cdots+r_{f-1}p^{f-1}$.

An alternative description of $\sigma(r)$, without the restriction $0\le r<q-1$, is given by
$$ \sigma(r)=(p-1)\sum_{i=1}^f\left\{{p^ir\over q-1}\right\}\;.$$
\end{theorem}

In order to rewrite $H_q$ when it is defined over $\bbbq$ we use the following result.

\begin{theorem}[Hasse-Davenport]\label{hassedavenport}
For any $N\in\bbbn$ dividing $\q$ we have
$$g(Nm)=-\omega(N)^{Nm}\prod_{j=0}^{N-1}{g(m+j\q/N)\over g(j\q/N)}\;.$$
\end{theorem}

This is Theorem 3.7.3 of \cite{cohen1}. Note the analogy with Euler's identity

$$\Gamma(Nz)=N^{Nz}{\prod_{j=0}^{N-1}\Gamma(z+j/N)\over
\prod_{j=1}^{N-1}\Gamma(j/N)}$$
for the $\Gamma$-function. The following proposition will be used
repeatedly. It is Fourier inversion on the multiplicative characters.

\begin{proposition}\label{fourier} Let $G:\fqstar\to\bbbc$ be a function. Then,
for any $\lambda\in\fqstar$ we have
$$G(\lambda)={1\over q-1}\sum_{m=0}^{q-2}G_m\omega(\lambda)^m\;,$$
where
$$G_m=\sum_{\lambda\in\fqstar}G(\lambda)\omega(\lambda)^{-m}\;.$$
\end{proposition}

Here is an application for later use.
\begin{lemma}\label{gaussproduct}
Let $\v a=(a_1,\ldots,a_n)\in\bbbz^n$. Define $a_{n+1}=-a_1-\cdots-a_n$ and
$a=\gcd(a_1,\ldots,a_n)$.
Then, for any integer $m$,
$$\sum_{v\in\bbbf_q,\v x\in(\fqstar)^n}\psi_q(v(1+x_1+\cdots+x_n))\omega(\v x^{\v a})^m
=(q-1)^n\delta(\v a m)+g(a_1m)\cdots g(a_{n+1}m)\;,$$
where we use the vector notations $\v x=(x_1,\ldots,x_n)$, $\v x^{\v a}=x_1^{a_1}\cdots x_n^{a_n}$ and
$\delta(\v x)=\delta(x_1)\cdots\delta(x_n)$ and $\delta(x)=1$ if $x\is0\mod{\q}$ and $0$ otherwise.
\end{lemma}

\begin{proof}{}We
carry out the summation over $x_1,\ldots,x_n$ and $v=0$ to get
$$\sum_{\v x\in(\fqstar)^n}\omega(\v x^{\v a})^m=(q-1)^n\prod_{i=1}^n \delta(a_im)=(q-1)^n\delta(\v a m)\;.$$
The summation over $x_1,\ldots,x_n$ and $v\in\fqstar$ yields
\begin{eqnarray*}
\sum_{v,x_i\in\fqstar}\psi_q(v(1+x_1+\cdots+x_n))\omega(\v x^{\v a})^m
&=&\sum_{v\in\fqstar}g(a_1m)\cdots g(a_nm)\psi_q(v)\omega(v^{(-a_1-\cdots-a_n)m})\\
&=&g(a_1m)\cdots g(a_{n+1}m).
\end{eqnarray*}
From these two summations our result follows.
\end{proof}

\section{Katz hypergeometric functions}\label{katzfunction}
Let $\boldalpha=(\alpha_1,\ldots,\alpha_d)$ and $\boldbeta=(\beta_1,\ldots,\beta_d)$ be
disjoint multisets in $\bbbq/\bbbz$. We assume that both $(q-1)\boldalpha$ and
$(q-1)\boldbeta$ are subsets of $\bbbz$. In \cite[p 258]{katz}
Katz considered for any $t\in\fqstar$ the exponential sum
$$Hyp_q(\boldalpha,\boldbeta|t)=
\sum_{(\v x,\v y)\in T_{t}}\psi_q(x_1+\cdots+x_d-y_1-\cdots-y_d)
\ \omega(\v x)^{\boldalpha\q}\omega(\v y)^{-\boldbeta\q}\;,$$
where $T_{t}$ is the toric variety defined by $tx_1\cdots x_d=y_1\cdots y_d$ and
$\omega(\v x)^{\boldalpha\q}$ stands for $\omega(x_1)^{\alpha_1\q}\cdots \omega(x_d)^{\alpha_d\q}$, etc.
Note that we have taken Katz's formula for the case $n=m=d$. Define
$$S_q(\boldalpha,\boldbeta|t)={1\over q-1}
\sum_{m=0}^{q-2}\prod_{i=1}^d\ g(m+\alpha_i\q)g(-m-\beta_i\q)
\ \omega((-1)^dt)^m\;.$$

\begin{proposition}[Katz]\label{exponentialsum}
With the above notation we have
$$Hyp_q(\boldalpha,\boldbeta|t)=\omega(-1)^{|\boldbeta|\q}S_q(\boldalpha,\boldbeta|t)$$
for all $t\in\fqstar$, where $|\boldbeta|=\sum_i\beta_i$.
\end{proposition}
\begin{proof}{}We consider
$Hyp_q(\boldalpha,\boldbeta|t)$ as function of $t$ and determine the $m$-th coefficient of its
Fourier expansion. Using Proposition \ref{fourier}, this reads
$${1\over q-1}\sum_{t,x_i,y_j\in\fqstar,tx_1\cdots x_d=y_1\cdots y_d}
\psi_q(x_1+\cdots+x_d-y_1-\cdots-y_d)
\ \omega(\v x)^{\boldalpha\q}\omega(\v y)^{-\boldbeta\q}\omega(t)^{-m}\;.$$
We substitute $t=(y_1\cdots y_d)(x_1\cdots x_d)^{-1}$ to get
$${1\over q-1}\sum_{x_i,y_j\in\fqstar}
\psi_q(x_1+\cdots+x_d-y_1-\cdots-y_d)
\ \omega(\v x)^{m+\boldalpha\q}\omega(\v y)^{-m-\boldbeta\q}\;.$$
Summation over the $2d$ variables $x_i,y_j$ then gives the desired result.
\end{proof}
\medskip

Notice that $Hyp_q$ (and $S_q$) are algebraic integers in the field
$\bbbq(\zeta_p,\zeta_{q-1})$, where $\zeta_n$ denotes a primitive $n$-th
root of unity. A simple consideration shows that $S_q$ is in
$\bbbq(\zeta_{q-1})$ if $\sum_{i=1}^d(\alpha_i-\beta_i)\in \bbbz$.
Moreover we have the following divisibility properties.

\begin{proposition}
For any integer $m$ the product $\prod_{i=1}^dg(m+\alpha_i\q)g(-m-\beta_i\q)$
is divisible by $g(|\boldalpha-\boldbeta|\q)$, where $|\boldalpha-\boldbeta|=
\sum_{i=1}^d(\alpha_i-\beta_i)$, and by
$\prod_{i=1}^dg((\alpha_i-\beta_i)\q)$. Moreover, the quotients are in $\bbbq(\zeta_{q-1})$.
\end{proposition}

This is a direct consequence of Lemma \ref{integralJ}.
Consequently $S_q(\boldalpha,\boldbeta|t)$ is divisible by these
numbers for any $t\in\fqstar$. This proposition suggests that we
might normalize $S_q$ by the factor $1/g(|\boldalpha-\boldbeta|\q)$
or by $1/\prod_ig((\alpha_i-\beta_i)\q)$. The latter normalization
is taken by John Greene in his definition of finite
hypergeometric sums, see \cite{greene}. More precisely, Greene uses the
function
$${\pm q^{1-d}\over \prod_ig((\alpha_i-\beta_i)\q)}S_q(\boldalpha,\boldbeta|t)\;.$$
We will adopt the normalization from the introduction. It keeps the symmetry in the $\alpha_i$ and
$\beta_j$, but accepts the possibility
that the function values may not be integers. Namely,

\begin{equation}\label{normalize}
H_q(\boldalpha,\boldbeta|t)=-\prod_{i=1}^d{1\over g(\alpha_i\q)g(-\beta_i\q)}
S_q(\boldalpha,\boldbeta|t)\;.
\end{equation}
To get a bound on the denominator of $H_q$ we introduce the {\it Landau function}
$$\lambda(\boldalpha,\boldbeta,x)=\sum_{i=1}^d\left(\{x+\alpha_i\}
-\{\alpha_i\}+\{-x-\beta_i\}-\{-\beta_i\}\right)\;,$$
where $\{x\}$ denotes the fractional part of the real number $x$.

\begin{theorem}\label{denominatorHqgeneral} Keep the above notation and let
$N$ be the common denominator of the $\alpha_i,\beta_j$. Define
$$\lambda=-\min_{x\in[0,1],\gcd(k,N)=1}\lambda(k\boldalpha,k\boldbeta,x)\;.$$
Then $q^{\lambda}H_q(\boldalpha,\boldbeta|t)$ is an algebraic integer.
\end{theorem}

\begin{proof}{}The 
definition of $Hyp_q$ and Proposition \ref{exponentialsum} imply that $S_q(\boldalpha,
\boldbeta|t)$ is an algebraic integer. Since $H_q$ is a normalization of $S_q$ by a
factor $\prod_{i=1}^dg(\alpha_i\q)g(\beta_i\q)$ the denominator of $H_q$ is a power of
$p$. We determine the $p$-adic valuation of the coefficients of the expansion of $H_q$.

Let $\p$ be defined as in Theorem \ref{stickelberger}. Then the $\p$-adic valuation
of the $m$-th coefficient of $H_q(\boldalpha,\boldbeta|t)$ equals
$$\sum_{j=1}^d\sum_{i=0}^{f-1}\left(\left\{p^i{m\over q-1}+\alpha_j\right\}+\left\{p^i{-m\over q-1}-\beta_j\right\}
-\{\alpha_j\}-\{-\beta_j\}\right)\;.$$
Note that this is greater than or equal to $-f\lambda$. Let $\sigma\in{\rm Gal}(\bbbq(\zeta_{q-1})/\bbbq)$ and $k$
such that $\sigma(\zeta_{q-1})=\zeta_{q-1}^k$. Then the $\sigma(\p)$-adic valuations of the
coefficients of $H_q(\boldalpha,\boldbeta|t)$ are equal to the
$\p$-adic valuations of the coefficients of
$H_q(k^*\boldalpha,k^*\boldbeta|t)$, where $kk^*\is1\mod{q-1}$.
The latter valuations are also bounded below by $-f\lambda$, so our
theorem follows.
\end{proof}
\medskip

We end by noting some obvious identities for $S_q$.
\begin{theorem}\label{functional1}
Let $\boldalpha,\boldbeta$ be as before.
Let $\mu\in\bbbq$ be such that $\mu\q\in\bbbz$ and denote
$(\alpha_1+\mu,\ldots,\alpha_d+\mu)$ by $\mu+(\alpha_1,\ldots,\alpha_d)$, etc.
For any $\lambda\in\fqstar$ we have
$$\omega(t)^{\mu\q}S_q(\mu+\boldalpha,\mu+\boldbeta|t)=S_q(\boldalpha,\boldbeta|t)\;.$$
Hence, by \fref{normalize},
$$\omega(t)^{\mu\q}H_q(\mu+\boldalpha,\mu+\boldbeta|t)=\prod_{i=1}^d
{g(\mu+\alpha_i\q)g(-\mu-\beta_i\q)\over g(\alpha_i\q)g(-\beta_i\q)}
H_q(\boldalpha,\boldbeta|t).$$
\end{theorem}

This follows directly from the definition of $S_q$ and the replacement of the summation variable
$m$ by $m-\mu\q$. We recall the analytic functions
$_2F_1(\alpha,\beta,\gamma|z)$ and $z^{1-\gamma}
\ _2F_1(\alpha+1-\gamma,\beta+1-\gamma,2-\gamma|z)$ (when $\gamma$ is not an integer)
which form a basis of solutions around $z=0$ for the Gauss hypergeometric equation.
Note that they can be considered as analogues of the functions in Theorem \ref{functional1}
with $\mu=(1-\gamma)\q$. Note that the latter differ by a factor independent of $t$.
In contrast with the finite case, the analytic functions are linearly independent over the
complex numbers.

We also have

\begin{theorem}\label{functional2}
Let $\boldalpha,\boldbeta$ be as before
and suppose $t\in\fqstar$. Then
$$S_q(\boldalpha,\boldbeta|t)=S_q(-\boldbeta,-\boldalpha|1/t)\;.$$
\end{theorem}

Again the proof follows directly from the definition of $S_q$ and replacing
the summation variable $m$ by $-m$.

\ifx
Some remarks on weight and Hodge polynomial. Suppose that the $\alpha_i,\beta_j$ are
in the interval $[0,1)$. Let $m_j=\#\{i|\beta_i<\alpha_j\}$ for $j=1,\ldots,d$.
Then, up to a power of $T$ the Hodge polynomial is given by
$$h(T)=\sum_{j=1}^d T^{m_j-j}\;.$$
The weight of the motive is $\max_j(m_j-j)-\min_j(m_j-j)-1$.
It follows from \cite[Thm 4.5]{beukersheckman} that the signature of the invariant
hermitian form is $|h(-1)|$. And of course $d=h(1)$.
\fi

\section{Hypergeometric motive over $\bbbq$}\label{overQ}
From now on we concentrate on the case when the Katz sums are defined over $\bbbq$.
In other words, as in the introduction, we assume the following:

\begin{assumption}
There exist natural numbers $p_1,\ldots,p_r,q_1,\ldots,q_s$ such that
$$\prod_{j=1}^d{X-e^{2\pi i\alpha_j}\over X-e^{2\pi i\beta_j}}
={\prod_{i=1}^r(X^{p_i}-1)\over \prod_{j=1}^s(X^{q_j}-1)}\ {\rm and}\ \gcd(p_1,\ldots,q_s)=1\;.$$
\end{assumption}

We now prove Theorem \ref{rewriteZ}, which tells us that we can rewrite our Katz sum
in a different shape.
\medskip

\begin{proof}{of Theorem \ref{rewriteZ}}We
use the notation given in the introduction.
Let $\delta$ be the degree of $D(X)$.
Suppose that the zeros of $D(X)$ are given by $e^{2\pi ic_j/\q}$ with $j=1,\ldots,\delta$, where possible repetitions of the roots are allowed. The coefficient
of $\omega(t)^m$ (without the $1/(1-q)$) in $H_q$ can be rewritten as
$$\omega(-1)^{dm}\left(\prod_{i=1}^r\prod_{j=0}^{p_i-1}{g(m+j\q/p_i)\over g(j\q/p_i)}\right)
\left(\prod_{i=1}^s\prod_{j=0}^{q_i-1}{g(-m-j\q/q_i)\over g(-j\q/q_i)}\right)
\prod_{j=1}^{\delta}{g(c_j)g(-c_j)\over g(m+c_j)g(-m-c_j)}\;.$$
Using Hasse-Davenport, the first product can be rewritten as
$$\prod_{i=1}^r-\omega(p_i)^{-p_im}g(p_im)\;.$$
Similarly, the second product is equal to
$$\prod_{i=1}^s-\omega(q_i)^{q_im}g(-q_im)\;.$$
To rewrite the third product, we use the evaluation $g(n)g(-n)=\omega(-1)^nq$
if $n\not\is0\mod{\q}$ and $g(0)g(0)=1$. We get
$$\omega(-1)^{\delta m}q^{-s(0)+s(m)}\;.$$
Finally notice that $d+\delta$ equals the degree of $\prod_{i=1}^s(X^{q_i}-1)$,
in particular $d+\delta=\sum_{i=1}^sq_i$. Hence the coefficient (without the factor
$1/(1-q)$) of $\omega(t)^m$ becomes
$$(-1)^{r+s}q^{-s(0)+s(m)}\prod_{i=1}^r g(p_im)\prod_{j=1}^s g(-q_jm)
\omega\left((-1)^{\sum_iq_i}{\prod_i q_i^{q_i}\over\prod_j p_j^{p_j}}\right)^m\;,$$
as asserted.
\end{proof}
\medskip

As we noted in the introduction the new summation makes sense for any choice of $q$ relatively
prime with the common denominator of the $\alpha_i,\beta_j$.  We now continue to work with
$H_q$ as defined by the new summation.

Before we proceed we can now be more specific about the arithmetic nature of $H_q$.
We need our main Theorem \ref{main} for this, which we will prove in
Section \ref{completion}.

\begin{theorem}\label{denominatorHq}
Suppose $H_q$ is defined over $\bbbq$. Define
$$\lambda=\min_{x\in[0,1]}\{p_1x\}+\cdots+\{p_rx\}+\{-q_1x\}+\cdots+\{-q_sx\}\;,$$
where as usual $\{x\}$ denotes the fractional part of $x$. Then $\lambda\ge1$ and
$$q^{\min(r,s)-\lambda}H_q(\boldalpha,\boldbeta|t)\in\bbbz\;.$$
\end{theorem}

In most cases it turns out that $\lambda-\min(r,s)$ is the exact $p$-adic
valuation of $H_q$. but we have not tried to prove this.
\medskip

\begin{proof}{}From Theorem \ref{main} it follows that the values of $H_q$ have a
denominator dividing $p^{\min(r,s)-1}$. To determine the $p$-adic valuation of $H_q$
we use Theorem \ref{stickelberger} and the notation therein.
To be more precise, we determine the $\p$-adic valuation of each coefficient
$g(p_1m)\cdots g(p_rm)g(-q_1m)\cdots g(-q_sm)$. Because $p_1+\cdots p_r-q_1-\cdots-q_s=0$
each such coefficient is in $\bbbq(\zeta_{q-1})$. According to Theorem \ref{stickelberger} its
$\p$-adic valuation is given by
$$\sum_{i=1}^f \left\{{p_1m\over q-1}\right\}+\cdots+\left\{{p_rm\over q-1}\right\}
+\left\{{-q_1m\over q-1}\right\}+\cdots+\left\{{-q_sm\over q-1}\right\}\;.$$
This is bounded below by $f\lambda$. Hence the $\p$-adic valuation of $H_q$ is bounded
below by $f\lambda$, and since the values of $H_q$ are in $\bbbq$ our theorem follows.
\end{proof}
\medskip

The main purpose of this paper is to show that the Katz hypergeometric function occurs
naturally in the point counting numbers of the completion and normalization
of the variety $V_{\lambda}$, which is defined by the points of $\bbbp^{r+s-1}$ satisfying
the equations

$$x_1+x_2+\cdots+x_r-y_1-\cdots-y_s=0,\qquad \lambda x_1^{p_1}\cdots x_r^{p_r}=y_1^{q_1}\cdots y_s^{q_s}$$
with $\lambda\in\fqstar$.

We have the following preliminary result.

\begin{proposition}\label{affinepointcount}
Assume that $\gcd(p_1,\ldots,p_r,q_1,\ldots,q_s)=1$.
Let $V_{\lambda}(\fqstar)$ be the set of points on
$V_{\lambda}$ with coordinates in $\fqstar$. Then
$$|V_{\lambda}(\fqstar)|={1\over q}(q-1)^{r+s-2}+
{1\over q(q-1)}\sum_{m=0}^{q-2}g(\v pm,-\v qm)\omega(\epsilon\lambda)^m\;,$$
where 
$$g(\v pm,-\v qm)=g(p_1m)\cdots g(p_rm)g(-q_1m)\cdots g(-q_sm)
\ {\rm and}\ \epsilon=(-1)^{\sum_iq_i}\;.$$
\end{proposition}

\begin{proof}{}We will use Lemma \ref{gaussproduct}. Let us take $n+1=r+s$ and write
$\v a=(a_1,a_2,\ldots,a_{n+1})=(p_1,\ldots,p_r,-q_1,\ldots,-q_s)$. In particular $n=r+s-1$.
We introduce the new variables $x_{r+j}=-y_j$ for $j=1,\ldots,s$. The equations for
$V_{\lambda}$ now read
$$x_1+x_2+\cdots+x_{n+1}=0,\qquad \lambda x_1^{a_1}\cdots x_{n+1}^{a_{n+1}}=\epsilon\;.$$
In vector notation the latter equation reads $\lambda \v x^{\v a}=\epsilon$. We dehomogenize
by setting $x_{n+1}=1$. Then
$$q|V_{\lambda}(\fqstar)|
=\sum_{v\in\bbbf_q,\v x\in(\fqstar)^n,\lambda\v x^{\v a}=\epsilon}\psi_q(v(1+x_1+\cdots+x_n))\;.$$
This is a function of $\lambda\in\fqstar$.
Let $\sum_{m=0}^{q-2}c_m\omega(\lambda)^m$ be its Fourier series.
The coefficient $c_m$ is given by
$$c_m={1\over q-1}\sum_{v\in\bbbf_q,\lambda,x_i\in\fqstar,\lambda
\v x^{\v a}=\epsilon}\psi_q(v(1+x_1+\cdots+x_n))
\ \omega(\lambda)^{-m}\;.$$
Substitute $\lambda=\epsilon\v x^{-\v a}$ to get
$$c_m={1\over q-1}\sum_{v\in\bbbf_q,x_i\in\fqstar}\psi_q(v(1+x_1+\cdots+x_n))
\ \omega(\epsilon\v x^{\v a})^m\;.$$
Then Lemma \ref{gaussproduct}, using $\gcd(a_1,\ldots,a_n)=1$, yields
$$c_m=(q-1)^{r+s-2}\delta(m)+{1\over q-1}\prod_{i=1}^{r+s} g(a_im)\omega(\epsilon)^m\;.$$
Putting everything together and dividing by $q$ gives us
$$|V_{\lambda}(\fqstar)|={1\over q}(q-1)^{r+s-2}+{1\over q(q-1)}
\sum_{m=0}^{q-2}g(a_1m)\cdots g(a_{r+s}m)\omega(\epsilon\lambda)^m\;,$$
which yields the desired statement.
\end{proof}
\medskip

Notice that most terms in the summation in Proposition \ref{affinepointcount} and
in the summation of Theorem \ref{rewriteZ} agree except for a few, which
differ by a factor which is a power of $q$. The final goal of
this paper is to show that this difference is caused by the difference between $V_{\lambda}$ and
its completion. In others words, the extra factors $q$ arise from addition of
the components of the completion of $V_{\lambda}$.
In the next section we shall elaborate on this idea and prove our main Theorem \ref{main}.

\section{Proof of the main theorem}\label{completion}

Let notation be as in the introduction. In order to complete the variety $V_{\lambda}$
given by equations (\ref{Vequation}), we need to fill in the points where one or more of the coordinates
are zero. In this section we shall do that quite explicitly, but first like to motivate
our approach by linking it to the theory of toric varieties. Our main reference
is the book by Cox, Little, Schenck, \cite{cls}. First of all, consider the variety $X_{\lambda}$
in $\bbbp^{r+s-1}$ defined by the projective equation
$\lambda\prod_{i=1}^r x_i^{p_i}=\prod_{j=1}^sy_j^{q_j}$. Take any solution $(a_1,\ldots,a_r;
b_1,\ldots,b_s)$ with non-zero coordinates. Then $x_i/a_i,y_i/b_i$ are the coordinates of
points on the projective variety $X_1$. The variety $X_1$ is a toric variety, the
points with non-zero coordinates form an open subvariety which is a torus (of rank $r+s-2$)
and the action of this torus can be extended to all points of $X_1$, see \cite[Def 3.1.1]{cls}.

To a toric variety one can associate a so-called fan (see \cite[Def 3.2.1]{cls}),
which we shall denote by $\Sigma$. To compute it
we use \cite[Thm 3.2.6]{cls} which gives a one-to-one correspondence between limits of group homomorphisms $\bbbc^{\times}\to X_1$ and the cones of a fan. Concretely, choose
$(\xi_1,\ldots,\xi_r;\eta_1,\ldots,\eta_s)\in\bbbr^{r+s}$ such that $\sum_{i=1}^rp_i\xi_i
=\sum_{j=1}^sq_j\eta_j$. Call this hyperplane $H$. We consider $H$ modulo shifts
with the vector $(1,1,\ldots,1)$. Any vector in $H/(1,1,\ldots,1)$ has a representing
element with $\min_{i,j}(\xi_i,\eta_j)=0$. Consider the embedding $\bbbr_{>0}^{\times}\to X_1$
given by $t\mapsto (t^{\xi_1},\ldots,t^{\xi_r};t^{\eta_1},\ldots,t^{\eta_s})$. Now let
$t\downarrow 0$. The limit consists of vectors with ones and zeros. 
A component is zero if and only if the corresponding $\xi_i$ or $\eta_j$ is positive. 
Two points in $H/(1,\ldots,1)$ lie in the same cone of $\Sigma$ if and only if their limits
are the same. Therefore, an open cone in $\Sigma$
is characterised by the index sets $S_x,S_y$ defined by $i\in S_x\iff \xi_i>0$
and $j\in S_y\iff \eta_j>0$.

It turns out that usually the toric variety $X_1$ corresponding to $\Sigma$
is highly singular.
We like to construct a partial blow up of $X_1$.
To do this one can construct a refinement $\Sigma'$ of $\Sigma$
and consider the toric variety corresponding to $\Sigma'$ (which may not be projective
ay more). The refinement we choose is a triangulation of $\Sigma$, which need not completely remove
the singularities from $X_1$, but it will be enough for our purposes.
Denote the blown up variety by $\hat{X}_1$.
Completeness of our blowup is garantueed because $\Sigma$ is a fan whose support is
all of $H/(1,\ldots,1)$, see \cite[Thm 3.1.19]{cls}.

Denote the corresponding completion of the translated variety $X_{\lambda}$
by $\hat{X}_{\lambda}$.
The variety $V_{\lambda}$ is the intersection of $\sum_{i=1}^rx_i=\sum_{j=1}^sy_j$ with the
affine part of $X_{\lambda}$ consisting of all points with non-zero coordinates.

\begin{definition}\label{completiondefinition}
Let notations be as above.
We define $\overline{V_{\lambda}}$ as the completion of $V_{\lambda}$ in $\hat{X}_{\lambda}$.
\end{definition}

We have seen above that the fan $\Sigma$ of $X_1$ has its support in $H/(1,\ldots,1)$.
We choose representatives of $H/(1,\ldots,1)$ in the set ${\cal H}$ defined by 
$$\xi_i,\eta_j\in\bbbr:\ \sum_{i=1}^r p_i\xi_i=\sum_{j=1}^sq_j\eta_j,
\ {\rm and}\ \min_{i,j}(\xi_i,\eta_j)=0.$$
The (closed) cones of $\Sigma$ are then characterized by sets $S_x\subset\{1,\ldots,r\}$ and
$S_y\subset\{1,\ldots,s\}$ via 
$i\not\in S_x\Rightarrow \xi_i=0$ and $j\not\in S_y\Rightarrow \eta_j=0$.

We now proceed to choose a refinement of $\Sigma$. Let $\overline{\cal H}$ be
the convex closure of ${\cal H}$. 
Note that it is the $\bbbr_{\ge0}$ span of
the vectors $\v v_{ij},\ i=1,\ldots,r;j=1,\ldots,s$ defined by 
$$\v v_{ij}=(0,\ldots,q_j,\ldots,0;0,\ldots,p_i,\ldots,0)\in\bbbr^{r+s},$$
where $q_j$ is at the $i$-th coordinate position of $\bbbr^r$ and
$p_i$ at the $j$-th position of $\bbbr^s$.

Let $\Delta^{r-1}$ be the standard simplex defined by $p_1\xi_1+\cdots+p_r\xi_r=1$
and $\xi_i\ge0$ and $\Delta^{s-1}$ the standard simplex defined by
$q_1\eta_1+\cdots+q_s\eta_s=1,\ \eta_j\ge0$. Then we
see that the points corresponding to the vectors $\v v_{ij}/(p_iq_j)$ are
vertices of a polytope isomorphic to $\Delta^{r-1}\times\Delta^{s-1}$. Thus
$\overline{\cal H}$ is a cone over the $(r+s-2)$-dimensional product simplex
$\Delta^{r-1}\times\Delta^{s-1}$.

A {\it simplicial subcone} of $\overline{\cal H}$ is the $\bbbr_{\ge0}$-span of any set of
vectors $\v v_{i_1j_1},\ldots,\v v_{i_{r+s-1}j_{r+s-1}}$ such that
$\{i_1,\ldots,i_{r+s-1}\}=\{1,\ldots,r\}$ and $\{j_1,\ldots,j_{r+s-1}\}=
\{1,\ldots,s\}$. By abuse of language we will call the corresponding set of
index pairs 
$$(i_1,j_1),\ldots,(i_{r+s-1},j_{r+s-1})$$
a {\it simplex}. We now triangulate $\overline{\cal H}$. That is, we write
$\overline{\cal H}$ as a union of simplicial cones whose interiors do not intersect. To do
this we take a triangulation of $\Delta^{r-1}\times\Delta^{s-1}$ and consider the
cones over the simplices of this triangulation. The triangulation we choose
is the so-called {\it staircase triangulation}, see \cite[Lemma 6.2.8]{lrs}.

\begin{proposition}A triangulation of $\overline{\cal H}$ is given by the set of cones over
all simplices $(i_1,j_1),\ldots,(i_{r+s-1},j_{r+s-1})$ with $i_1\le i_2\le\cdots\le i_{r+s-1}$
and $j_1\le j_2\le\cdots\le j_{r+s-1}$.
\end{proposition}

Thus simplices of the staircase triangulation are in one-to-one correspondence with
lattice paths in the plane that go from $(1,1)$ to $(r,s)$ by making steps of size $1$ in north
or east direction.

To illustrate this proposition we give the staircase triangulation for $\Delta^2\times \Delta^1$,
which is a triangular prism. In this case $r=3$, $s=2$. The sequences $(i_1,j_1),\ldots,
(i_4,j_4)$ are depicted as dots in the following pictures,

\centerline{\includegraphics[height=2cm]{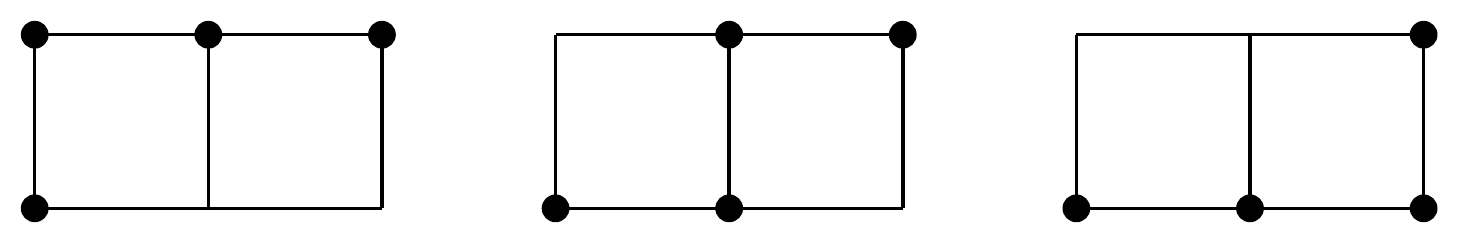}}

There are three possible pictures which corresponds to the fact that a triangulation of
a triangular prism consists of three simplices.

Any subset of a simplex is called a {\it cell}. The cone spanned by the corresponding vectors
$\v v_{ij}$ is called a {\it cellular cone}. From now on we restrict to cells which are subset
of the simplices belonging to the staircase triangulation. They are characterized by
sequences $(i_1,j_1),\ldots,(i_l,j_l)$ with $i_1\le\cdots\le i_l$ and $j_1\le\ldots\le j_l$, but
$\{i_1,\ldots,i_l\}$ need not equal $\{1,2,\ldots,r\}$ anymore and similarly for $\{j_1,\ldots,j_l\}$.
The corresponding cellular cones have dimension $l$ and can be pictured in a similar way as above.
To any cell $C$ we can associate the supports $S_x(C)=\{i_1,\ldots,i_l\}$ and
$S_y(C)=\{j_1,\ldots,j_l\}$. Let $S(C)$ be the disjoint union of $S_x$ and $S_y$.
Then clearly, $l+1\le |S(C)|\le\min(r+s,2l)$. We say that the support of $C$ is {\it maximal} if
$|S(C)|=r+s$. The cellular cone corresponding to a maximal cel is not contained in ${\cal H}$,
the cone corresponding to a non-maximal cel is contained in ${\cal H}$.
Furthermore, if $S(C)$ is not maximal, then  $|S(C)|\le r+s-2$, since points
with only one non-zero coordinate cannot exist in ${\cal H}$ because of the equation
$\sum_ip_i\xi_1=\sum_jq_j\eta_j$. 

\begin{definition}
The fan $\Sigma'$ is defined by all
cellular cones whose support is not maximal.
\end{definition}

Choose a cell $C$ given by $(i_1,j_1),\ldots,(i_l,j_l)$. To it we shall associate a component,
denoted by $W_{C,\lambda}$, of the completion of $V_{\lambda}$.
Choose an index $\sigma$ not in $S_x$, or if that is not possible, $\sigma\not\in S_y$.
Assume the first case happens. Then set $x_{\sigma}=1$ (dehomogenization). However,
for the sake of elegance, we continue to write $x_{\sigma}$. 

We replace the coordinates $x_i,i\in S_x,y_j,j\in S_y$ by monomials
in $|S|$ new variables, namely, $u_1,\ldots,u_l, w_1,\ldots,w_{|S|-l-1},z$. It could happen that
there are no variables $w_i$, namely if $|S|$ has the minimal value $l+1$. We begin by choosing
$l$ lattice points $\v u_1,\ldots,\v u_l$ in the $\bbbr_{\ge0}$-span of
$\v v_{i_1j_1},\ldots,\v v_{i_lj_l}$ such that
they form a basis of the lattice points in the $\bbbr$-span of
$\v v_{i_1j_1},\ldots,\v v_{i_lj_l}$. 
We next extend this basis to a basis of the lattice points in the space
$p_1\xi_1+\cdots+p_r\xi_r=q_1\eta_1+\cdots+q_s\eta_s$
with support in $S$. That is, for any such lattice point we have $\xi_i=0$
if $i\not\in S_x$ and $y_j=0$ if
$j\not\in S_y$. The additional vectors are denoted by $\v w_1,\ldots,\v w_{|S|-l-1}$.
Finally we extend this
basis to a basis of all lattice points with support in $S$ by adding a suitable vector $\v z$.

Let us denote the $m$-th component of $\v u_1$ by $u_{1i_1}$, and similarly for the other indices.
Now replace $x_{i_1}$ in equation (\ref{Vequation}) by 
$$u_1^{u_{1i_1}}u_2^{u_{2i_1}}\cdots u_l^{u_{li_1}}w_1^{w_{1i_1}}\cdots z^{z_{1i_1}}$$
and similarly for the other variables $x_m,m\in S_x,y_m,m\in S_y$. It is important to notice that the
minimum of the $i_1$-th components of $\v u_1,\ldots,\v u_l$ is positive. The component
$W_{C,\lambda}$ to
be added consists of  the points with $u_1=\cdots=u_l=0$ and all other coordinates non-zero.
Setting $u_1=u_2=\cdots=u_l=0$ in our equation will render $x_i,i\in S_x,y_j\in S_y$ zero. Furthermore, since the $\v u_i$ and $\v w_i$ all satisfy the
equation $p_1\xi_1+\cdots+p_r\xi_r=q_1\eta_1+\cdots+q_s\eta_s$ the variables $u_i,w_i$ will
be absent from the equation $\lambda x_1^{p_1}\cdots x_r^{p_r}=y_1^{q_1}\cdots y_s^{q_s}$
after we made the substitution. The variable $z$ will occur in the equation with exponent $\pm a_S$
where $a_S=\gcd(p_{i_1},\ldots,p_{i_l},q_{j_1},\ldots,q_{j_l})$. 

After the variable substitution and setting $u_1=\cdots=u_l=0$ we are left with the equations
\begin{equation}\label{Wequation}
\sum_{i\not\in S_x}x_i=\sum_{j\not\in S_y}y_j,
\quad \lambda z^{a_S}\prod_{i\not\in S_x}x_i^{p_i}=\prod_{j\not\in S_y}y_j^{q_j}.
\end{equation}
Notice that the variables $w_i$ have disappeared from the equations. We choose them arbitrarily
and non-zero. The variables $x_i,i\not\in S_x$ and $y_j,j\not\in S_y$ are also taken
non-zero. So the component $W_{C,\lambda}$ is $\bbbg_m^{|S|-l-1}$ times the affine variety
of points with non-zero coordinates given by the equations (\ref{Wequation}). 

Before giving the point count on $W_{C,\lambda}$ we present an example of the
evaluation of some components. Consider the case $p_1=6,p_2=3,p_3=2,p_4=1,q_1=8,q_2=q_3=4$.
We have $r=4,s=2$ and the equations of $V_{\lambda}$ read
\begin{equation}\label{Vexample}
x_1+x_2+x_3+x_4=y_1+y_2,\quad \lambda x_1^6x_2^3x_3^2x_4=y_1^8y_2^4.
\end{equation}
We would like to construct the components $W_{C,\lambda}$ for the following
cells $C$:

\centerline{\includegraphics[height=1.5cm]{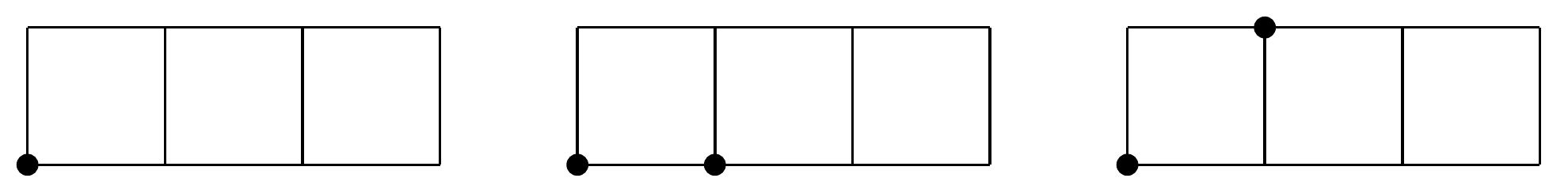}}

In all three examples we shall set $x_4=1$.
First consider the cell given by $(1,1)$. The relevant variables that become $0$ are $x_1,y_1$.
Choose $\v u_1=(4,3)$ (actually $\v u_1\in\bbbr^6$, but we write down only the non-zero coordinates
at the positions corresponding to $x_1,y_1$). We choose $\v z=(1,1)$.
This gives rise to the substitution $x_1=u_1^4z,y_1=u_1^3z$. After this substitution equations 
(\ref{Vexample}) become
$$u_1^4z+x_2+x_3+1=u_1^3z+y_2,\quad \lambda x_2^3x_3^2=z^2y_2^4.$$
After setting $u_1=0$ we get
$$x_2+x_3+1=y_2,\qquad \lambda x_2^3x_3^2=z^2y_2^4.$$

Now consider the cell given by $(1,1),(2,1)$. The relevant variables that become $0$ are $x_1,x_2,y_1$.
For $\v u_1,\v u_2$ we have to find a lattice basis in the cone spanned by $(4,0,3),(0,8,3)$.
Choose $\v u_1=(4,0,3),\v u_2=(3,2,3)$. Finally we choose $\v z=(0,3,1)$. Our substitution becomes
$x_1=u_1^4u_2^3,x_2=u_2^2z^3,y_1=u_1^3u_2^3z$. Equations (\ref{Vexample}) become
$$u_1^4u_2^3+u_2^2z^3+x_3+1=u_1^3u_2^3z,\quad \lambda zx_3^2=y_2^4.$$
Setting $u_1=u_2=0$ leaves us with
$$x_3+1,\qquad \lambda zx_3^2=y_2^4.$$
Now consider the cell given by $(1,1),(2,2)$. That means $x_1,x_2,y_1,y_2$ are the relevant variables
that go to $0$. We choose $\v u_1=(4,0,3,0),\v u_2=(0,4,0,3)$. Since $|S|-l-1=1$ in this case we can choose
a vector $\v w_1$, say $\v w_1=(1,-2,1,-2)$. Note that all three vectors satisfy
$6t_1+3t_2=8t_3+4t_4$. Finally we take $\v z=(0,1,0,1)$. This yields the substitution
$$x_1=u_1^4w_1,\ x_2=u_2^4w_1^{-2}z,\ y_1=u_1^3w_1,\ y_2=u_2^3w_1^{-2}z$$
and equations (\ref{Vexample}) become
$$u_1^4w_1+u_2^4w_1^{-2}z+x_3+1=u_1^3w_1+u_2^3w_1^{-2}z,\quad
\lambda x_3^2=z.$$
Setting $u_1=u_2=0$ yields $x_3+1=0,\lambda x_3^2=z$. We see that $w_1$ has disappeared from the equations.

\begin{proposition}
The points of $W_{C,\lambda}$ are non-singular.
\end{proposition}

\begin{proof}{}Let
$V$ be an affine variety given by two equations 
$$F(x_1,\ldots,x_n)=G(x_1,\ldots,x_n)=0\;.$$
Then a point $P\in V$ is
non-singular if and only if the vectors $(F_1(P),\ldots,F_n(P))$ and $(G_1(P),\ldots,G_n(P))$ are dependent. 
Here $F_i,G_i$ denotes differentiation of $F$, resp $G$, with respect to $x_i$. For $F$ and $G$ we take the equations (\ref{Vequation})
with $x_{\sigma}=1$ and where we have made the necessary change of variables in the construction of $W_{C,\lambda}$.
Let us say $F$ comes from the linear equation. Then $F$ has the form $\Sigma+ 1+\sum_{i\not\in S_x\cup\sigma}x_i=
\sum_{j\not\in S_y}y_j$, where $\Sigma$ is the sum of all terms containing powers of the $u_j$.
The polynomial $G$ has the form $\lambda z^a\prod_{i\not\in S_x\cup\sigma}x_i^{p_i}=\prod_{j\not\in S_y}y_j^{q_j}$.
Choose a point $P\in W(C\lambda)$ and consider the derivates of $F,G$ at that point. In particular the derivative of $F$ with respect
to $z$ is $0$, whereas $G_z(P)$ is $az^{a-1}\prod_ix_i^{p_i}\ne0$. If the vectors of derivations were dependent we conclude
that all derivatives of $F$ are zero at $P$. However, since $|S|\le r+s-2$,
there exists $i\not\in S_x\cup\sigma$ or $j\not\in S_y$ and
the derivative of $F$ with respect to the corresponding variable is $\pm1$. So we have a contradiction and $P$ is non-singular.
\end{proof}

\begin{proposition}\label{componentcount}
Let $C$ be a cell as above, $l$ its number of elements and
$S$ its support. Let $W_{C,\lambda}$ be the product of $\bbbg_m^{|S|-l-1}$
and the variety defined by the equations (\ref{Wequation}). Then
$$
|W_{C,\lambda}(\fqstar)|={1\over q}(q-1)^{r+s-l-2}+\\
{(-1)^{|S|}\over q}(q-1)^{|S|-l-1}\sum_{m=0}^{q-2}
\delta(a_Sm)g(\v pm,-\v qm)\omega(\epsilon\lambda)^m,
$$
where
$$g(\v pm,-\v qm)=\prod_{i=1}^rg(p_im)\prod_{j=1}^sg(-q_jm)\ {\rm and}\ \epsilon=(-1)^{\sum_{j=1}^sq_j}.$$
In addition, if we take $C$ empty, then $|S|=l=0$ and the point count coincides with
Proposition \ref{affinepointcount} if we set $a_S=0$.
\end{proposition}

\begin{proof}{}We work with a non-empty cell. First we count the number $N$ of solutions of
equations (\ref{Wequation}) and then multiply the result by $(q-1)^{|S|-l-1}$.
Let us rewrite the equation (\ref{Wequation}) in the more managable form 
$$1+x_1+\cdots+x_n=0,\qquad \epsilon'=\lambda z^{a_S}\prod_{i=1}^nx_i^{a_i}.$$
To arrive at this we set $x_{\sigma}=1$ and replace $-y_j$ by variables $x_m$
and the set $\{a_1,\ldots,a_n\}$ to be the set consisting of $p_i,i\not\in S_x\cup\sigma$
and $-q_j,j\not\in S_y$. Finally, $\epsilon'=(-1)^{\sum_{j\in S_y}q_j}$. Notice
that $n=r+s-|S|-1$.
 
De cardinality $N$ is computed using
$$qN=\sum_{z,x_i\in\fqstar,v\in\bbbf_q,\epsilon'=\lambda z^{a_S}\prod_ix_i^{a_i}}
\psi_q(v(1+\sum_ix_i))\;,$$
where the summation and product with index $i$ are to be taken over $i=1,2,\ldots,n$.
We determine the Fourier coefficient $c_m$ of this expression.
$$c_m={1\over q-1}\sum_{z,x_i,\lambda\in\fqstar,v\in\bbbf_q,\epsilon'=\lambda z^{a_S}\prod_ix_i^{a_i}}
\psi_q(v(1+\sum_ix_i))\ \omega(\lambda)^{-m}\;.$$
Substitute $\lambda=\epsilon' z^{-a_S}\prod_ix_i^{-a_i}$ to get
$$c_m={1\over q-1}\sum_{z,x_i\in\fqstar,v\in\bbbf_q}
\psi_q(v(1+\sum_ix_i))\ \omega (z^{a_S}\prod_ix_i^{a_i})^m\omega(\epsilon')^m\;.$$
Summation over $z$ yields the factor $(q-1)\delta(a_Sm)$. Then, using Lemma
\ref{gaussproduct}, summation over $v$ and the $x_i$ yields
$$c_m=\delta(a_Sm)\left((q-1)^{r+s-|S|-1}\delta(am)
+{1\over q-1}g(-\sum_ia_im)\prod_{i}g(a_im)\right)\omega(\epsilon')^m\;,$$
where $a=\gcd_{i=1,\ldots,n}(a_i)$.

We rewrite this expression for $c_m$ in terms of our original data.
First of all $a$ is the gcd of all $p_i,q_j$ with $i\not\in S_x\cup\sigma$ and
$j\not\in S_y$. Note that any divisor $d$ of both $a$ and $a_S$ divides all
$p_i,q_j$ except possibly $p_{\sigma}$. Because the sums of all $p_i$ 
equals the sum of all $q_j$, $d$ also divides $p_{\sigma}$. Since
the gcd of all $p_i,q_j$ is $1$ we conclude that $\gcd(a,a_S)=1$, hence
$\delta(am)\delta(a_Sm)=\delta(m)$. 

Secondly, if $a_Sm\is0\mod{\q}$ then
$p_im\is0\mod{\q}$, hence $g(p_im)=-1$, for all $i\in S_x\cup\sigma$ and similarly
$g(-q_jm)=-1$ for all $j\in S_y$. Furthermore,
$-\sum_ia_im\is-\sum_{i=1}^rp_im+\sum_{j=1}^sq_jm+p_{\sigma}m\mod{\q}\is p_{\sigma}m\mod{\q}$. Hence
$$\delta(a_Sm)g(-\sum_ia_im)\prod_ig(a_im)=(-1)^{|S|}
\delta(a_Sm)\prod_{i=1}^rg(p_im)\prod_{j=1}^sg(-q_jm)\;.$$

Thirdly, if $\delta(a_Sm)=1$, then $m\sum_{j\in S_y}q_j$ is divisible by $a_Sm$, hence divisible
by $q-1$. If $q$ is odd this means that $m\sum_{j\not\in S_y}q_j\is m\sum_{j=1}^sq_j\mod{2}$.
hence $(\epsilon')^m=\epsilon^m$. If $q$ is even we work in characteristic $2$ and thus
$1=-1$. We obtain
$$c_m=(q-1)^{r+s-|S|-1}\delta(m)+
(-1)^{|S|}\delta(a_Sm)\prod_{i=1}^rg(p_im)\prod_{j=1}^sg(-q_jm)\omega(\epsilon)^m.$$
Division by $q$, multiplication by $(q-1)^{|S|-l-1}\omega(\lambda)^m$,
and summation over $m$ yields the desired result.
\end{proof}
\medskip

Denote by ${\cal T}_{rs}$ the set of all cells of the staircase triangulation of
$\Delta^{r-1}\times \Delta^{s-1}$.
Recall that any such cell $C$ of ${\cal T}_{rs}$ is
given by a sequence $(i_1,j_1),\ldots,(i_l,j_l)$ of distinct pairs such
that $i_{m+1}-i_m,j_{m+1}-j_m\ge0$ for $m=1,\ldots,l-1$. The number $l=l(C)$ is called the
length of the cell. As before, the support $S(C)$ is the disjoint union of the index sets
$\{i_1,\ldots,i_l\},\{j_1,\ldots,j_l\}$. Notice that the difference
$|S(C)|-l(C)-1$ is equal to the number of indices $m$ such that $i_{m+1}-i_m,j_{m+1}-j_m>0$.
We note specifically that the sequence may be empty, i.e. $l=0$, in which case we speak
of the empty cell.

\begin{proposition}\label{summationterm}
With the above notation we have
$$\sum_C (q-1)^{|S(C)|-l(C)}(-1)^{|S(C)|}=q^{\min(r,s)}\;,$$
where the summation extends over all cells $C$ of ${\cal T}_{rs}$, including the
empty one.
\end{proposition}

\begin{proof}{}
We divide the cells into two types. The set $A$ of cells for which there exists
$m$ such that $i_m\ne j_m$ and the set $B$ of cells for which $i_m=j_m$ for all
$m$, including the empty one.
We show that the total contribution in the sum coming from the cells in $A$
is zero. Denote by $p$ the minimum of all
$\min(i_m,j_m)$ for which $i_m\ne j_m$. We distinguish two types in the set $A$.
\begin{itemize}
\item[(i)] The point $(p,p)$ is contained in the cell.
\item[(ii)] The point $(p,p)$ is not contained in the cell.
\end{itemize}
We can map the cells from $A(i)$ one-to-one to $A(ii)$ by deletion of the
point $(p,p)$. In the process the difference $|S(C)|-l(C)$ does not change but the
value of $|S(C)|$ changes by one. Hence the terms that are paired by this
mapping cancel. So it remains to compute the contribution from the cells
of $B$. They are all of the form $(i_1,i_1),\ldots,(i_l,i_l)$ and the
length is at most $\min(r,s)$. For each cell in $B$ the number $|S|=2l$ is
even and also, $|S|-l=l$. The number of cells in $B$ of length $l$
equals ${\min(r,s)\choose l}$. Hence the sum reads
$$\sum_{l=0}^{\min(r,s)}{\min(r,s)\choose l}(q-1)^{l}=(q-1+1)^{\min(r,s)}=q^{\min(r,s)}\;,$$
as asserted.
\end{proof}

\begin{proposition}\label{summationmain}
With the above notations we have
$$\sum_C (q-1)^{r+s-l(C)-1}=\sum_{m=0}^{\min(r-1,s-1)}{r-1\choose m}{s-1\choose m}q^{r+s-m-1}\;,$$
where the summation extends over all cells $C$ of ${\cal T}_{rs}$, including the
empty one.
\end{proposition}

\begin{proof}{}
Denote the sum on the left by $A_{rs}$ and the sum on the right by $B_{rs}$.
When $r=1$ the number of cells
of length $l$ is equal to ${s\choose l}$. Hence
$$A_{1s}=\sum_{l=0}^s{s\choose l}(q-1)^{1+s-l-1}=q^s\;,$$
which equals $B_{1s}$. is the desired answer when $r=1$. Similarly we have
$A_{r1}=B_{r1}$ for all $r$.

We claim that for all $r,s>1$,
$$A_{rs}=qA_{r-1,s}+qA_{r,s-1}-(q^2-q)A_{r-1,s-1}\;.$$
The summation over the cells of ${\cal T}_{rs}$ consists of the cells
in ${\cal T}_{r-1,s}$, the cells ${\cal T}_{r,s-1}$ and the cells
with $(i_l,j_l)=(r,s)$. The contribution to
$A_{rs}$ from ${\cal T}_{r-1,s}$ equals $(q-1)A_{r-1,s}$ because
$k$ is increased by $1$. Similarly we get a contribution from ${\cal T}_{r,s-1}$.
However, cells in ${\cal T}_{r-1,s-1}$ have
been counted doubly. So we have to subtract the latter contribution once, which
equals $(q-1)^2A_{r-1,s-1}$. This gives the contribution
$$(q-1)A_{r-1,s}+(q-1)A_{r,s-1}-(q-1)^2A_{r-1,s-1}\;.$$
The cells with endpoint $(r,s)$ have
not been counted yet. They arise by adding $(r,s)$ to the cells
in ${\cal T}_{r-1,s-1},{\cal T}_{r-1,s}$ and ${\cal T}_{r,s-1}$. The length of the cell increases
by $1$ and $k$ increases by 1, except with ${\cal T}_{r-1,s-1}$ in which case $k$
increases by $2$. Using inclusion-exclusion we arrive at a contribution
$$A_{r-1,s}+A_{r,s-1}-(q-1)A_{r-1,s-1}\;.$$
Our claim now follows. We see that the $A_{rs}$ are uniquely determined
by the recursion and $A_{r1}=q^r,A_{1s}=q^s$. 

It remains to verify
that $B_{rs}$ satisfies the same recursion.
Consider $B_{rs}-qB_{r-1,s}-qB_{r,s-1}+(q^2-q)B_{r-1,s-1}$ for $r,d\ge1$.
It equals
\begin{eqnarray*}
&&\sum_{m\ge0}\left({r-1\choose m}{s-1\choose m}-{r-2\choose m}{s-1\choose m}
-{r-1\choose m}{s-2\choose m}+{r-2\choose m}{s-2\choose m}\right)q^{r+s-m-1}\\
&&-{r-2\choose m}{s-2\choose m}q^{r+s-m-2}.
\end{eqnarray*}
where we use the convention that we sum over all $m\ge0$ and use ${a\choose m}=0$
if $a$ is an integer $<m$. 
The terms of the sum on the first line are easily seen to equal
${r-2\choose m-1}{s-2\choose m-1}q^{r+s-m-1}$ if $m\ge1$ and $0$ if $m=0$. So we
are left with
$$\sum_{m\ge1}{r-2\choose m-1}{s-2\choose m-1}q^{r+s-m-1}
-\sum_{m\ge0}{r-2\choose m}{s-2\choose m}q^{r+s-m-2}\;,$$
which equals $0$. Hence $B_{rs}$ satisfies the recurrence and the equality $A_{rs}=B_{rs}$ holds
for all $r,s$.
\end{proof}

\begin{proposition}\label{summationmaximal}
With the above notation we have
$$\sum_{S(C)=r+s} (q-1)^{r+s-l(C)-1}=\sum_{m=0}^{\min(r-1,s-1)}{r-1\choose m}{s-1\choose m}q^m\;,$$
where the summation extends over all cells $C$ of ${\cal T}_{rs}$ with $|S(C)|=r+s$.
\end{proposition}

We call cells with $|S(C)|=r+s$ {\it maximal cells}. They correspond to the 
cellular cones that are not contained in ${\cal H}$.
Although these maximal cells do not contribute to the pointcount on
$\overline{V_{\lambda}}$ it is convenient to include them in summations over the
cells of ${\cal T}_{rs}$.
\medskip

\begin{proof}{}The cells $C$ with $|S(C)|=r+s$ are given by $(i_1,j_1),\cdots,(i_l,j_l)$
with $(i_{m+1},j_{m+1})-(i_m,j_m)\in\{(1,0),(0,1),(1,1)\}$ for all $m$
and $(i_1,j_1)=(1,1)$ and $(i_l,j_l)=(r,s)$.
Let $D_{rs}$ be the corresponding
sum. We claim that
$$D_{rs}=D_{r-1,s}+D_{r,s-1}+(q-1)D_{r-1,s-1}$$
for all $r,s>1$. This is a consequence of the fact that maximal cells in ${\cal T}_{rs}$
arise by adding $(r,s)$ to a maximal cell in ${\cal T}_{r-1,s}$,
${\cal T}_{r,s-1}$ or ${\cal T}_{r-1,s-1}$. It is clear that $D_{1s}=D_{r1}=1$.
Together with the recurrence and an argument as in the previous proposition
we arrive at our assertion.
\end{proof}

According to Proposition \ref{componentcount} we associate to any cell $C$ the counting number
$$
N(C)={1\over q}(q-1)^{r+s-l(C)-2}+{(q-1)^{|S(C)|-l(C)}\over q(q-1)}\sum_{m=0}^{q-2}
(-1)^{|S(C)|}\delta(a_{S(C)}m)g(\v pm,-\v qm)\omega(\epsilon\lambda)^m\;,$$
where
$$g(\v pm,-\v qm)=\prod_{i=1}^rg(p_im)\prod_{j=1}^sg(-q_jm).$$
In particular, if $|S(C)|=l(C)=0$ ($C$ empty), then this number coincides
with $|V_{\lambda}(\fqstar)|$ according to Proposition \ref{affinepointcount}.
To obtain the cardinality of $\overline{V_{\lambda}}(\bbbf_q)$, we take the sum of $N(C)$ over
all cells $C$ with $|S(C)|\le r+s-2$. A straightforward check shows that our definition
gives $N(C)=0$ if $|S(C)|=r+s-1$. Hence we can take the sum of $N(C)$ over all cells $C$ and subtract
the contribution from the maximal cells.

We first sum over all $C$ and do this term by term in the above expression for $N(C)$.
The summation of ${1\over q}(q-1)^{r+s-l(C)-2}$ yields, according to Proposition
\ref{summationmain},
$${1\over q-1}\sum_{m=0}^{\min(r-1,s-1)}{r-1\choose m}{s-1\choose m}q^{r+s-m-2}\;.$$

Now choose the $m$-th term in the Fourier sum. Let $I(m)$ be the set of $i$ such that $p_im\is0
\mod{\q}$ and $J(m)$ the set of $j$ with $q_jm\is0\mod{\q}$.
Since $\delta(a_{S(C)}m)$ is only non-zero if
$S_x(C)\subset I(m)$ and $S_y(C)\subset J(m)$, we must sum over all such $C$. 
According to Proposition \ref{summationterm} the sum of $(-1)^{S(C)}(q-1)^{|S(C)|-l(C)}$ over such cells is equal
to $q^{\min(|I(m)|,|J(m)|)}$. The total contribution in the sum of the $N(C)$ reads
$${1\over q-1}q^{\min(|I(m)|-1,|J(m)|-1)}g(p_1m)\cdots g(p_rm)g(-q_1m)\cdots g(-q_sm)\;.$$
Note that $\min(|I(m)|,|J(m)|)$ equals $s(m)$, which occurs in Theorem \ref{rewriteZ}.

We now subtract the contributions coming from the maximal cells. If $C$ is maximal,
then $a_{S(C)}=1$ since the gcd of all $p_i,q_j$ is $1$.
Hence $\delta(a_{S(C)}m)=1$ if and only if $m=0$, and $0$ otherwise.
So, for a maximal cell $C$ we get
$$N(C)={(q-1)^{r+s-l(C)-1}\over q}((q-1)^{-1}+1)=(q-1)^{r+s-l(C)-2}\;.$$
According to Proposition \ref{summationmaximal} summation over these terms
gives
$${1\over q-1}\sum_{m=0}^{\min(r-1,s-1)}{r-1\choose m}{s-1\choose m}q^m\;.$$
Together these contributions yield
$$\sum_{m=0}^{\min(r-1,s-1)}{r-1\choose m}{s-1\choose m}{q^{r+s-m-2}-q^m\over q-1}
+{1\over q-1}\sum_{m=0}^{q-2}q^{s(m)-1}g(\v pm,-\v qm)\omega(\epsilon\lambda)^m\;.$$
Combination of this result with Theorem \ref{rewriteZ} yields Theorem \ref{main}.

\section{Some alternative varieties}
Let $k=r+s$. Notice that the dimension of the variety $V_{\lambda}$ equals $k-3$.
In some cases this
is higher than expected. For example, if $\boldalpha=\{1/2,1,2\}$ and
$\boldbeta=\{1,1\}$ one would get $p_1=p_2=2,q_1=q_2,q_3=q_4=1$.
So $k=6$ and the dimension of $V_{\lambda}$ equals $3$. However, one would expect
that the Legendre family of elliptic curves is associated to this particular choice
of $\boldalpha$ and $\boldbeta$. One can remedy this by noticing that the
set $\{2,2,-1,-1,-1,-1\}$ can be divided into two sets $\{2,-1,-1\}$ both of
whose elements sum up to $0$. In such cases the dimension of $V_{\lambda}$ can be
reduced by considering an alternative variety.

Let us write $(a_1,\ldots,a_k)=(p_1,\ldots,p_r,-q_1,\ldots,-q_s)$.
Suppose $\{1,\ldots,k\}$ is a union of disjoint subsets $K_1,\ldots,K_l$ such
that $\sum_{r\in K_i}a_r=0$ for $i=1,\ldots,l$. Let $A_i=(a_j)_{j\in K_i}$ for $i=1,\ldots,l$.
Define the variety ${\cal V}_{\lambda}$ as the subvariety of $\bbbp^{|K_1|-1}\times
\cdots\times\bbbp^{|K_l|-1}$ given by the equations
$$\sum_{i\in K_1}x_i=\cdots=\sum_{i\in K_l}x_i=0,\qquad \lambda x_1^{a_1}\cdots x_k^{a_k}=\epsilon\;.$$
Note that ${\cal V}_{\lambda}$ has dimension $k-2l-1$. In our particular example we get
dimension one.

\begin{theorem}\label{legendrecount} We use the above notation. In addition we assume that
the gcd of the elements $\{a_i|i\in K_j\}$ is one for all $j=1,\ldots,l$. Then
$$|{\cal V}_{\lambda}(\fqstar)|={1\over q-1}\prod_{j=1}^l Q_{|K_j|}(q)
+{1\over q^l(q-1)}\sum_{m=1}^{q-2}g(a_1m)\cdots g(a_km)
\ \omega(\epsilon\lambda)^m\;,$$
where $Q_r(x)=((x-1)^{r-1}+(-1)^r)/x$.
\end{theorem}

\begin{proof}{}For each $K_i$ we choose $k_i\in K_i$ and set $x_{k_i}=1$ (dehomogenization).
We introduce the short-hand notation $\Sigma_j=1+\sum_{i\in K_j,i\ne k_j}x_i$. Then
$$q^l|{\cal V}_{\lambda}(\fqstar)|=
\sum_{u_1,\ldots,u_l\in\bbbf_q,\v x\in(\fqstar)^k,\lambda\v x^{\v a}=1}
\psi_q(u_1\Sigma_1+\cdots+u_l\Sigma_l)\;.$$
The latter sum is a function of $\lambda\in\fqstar$ and the $m$-th coefficient of its
Fourier series reads
$$c_m={1\over q-1}\sum_{u_1,\ldots,u_l\in\bbbf_q,\lambda,x_i\in\fqstar,\lambda\v x^{\v a}=1}
\psi_q(u_1\Sigma_1+\cdots+u_l\Sigma_l)\ \omega(\lambda)^{-m}\;.$$
In the summation over $\v x$ we let $x_{k_i}=1$ for $i=1,\ldots,l$.
Substitution of $\lambda=\epsilon\v x^{-\v a}$ yields
$$c_m={1\over q-1}\sum_{u_1,\ldots,u_l\in\bbbf_q,x_i\in\fqstar}
\psi_q(u_1\Sigma_1+\cdots+u_l\Sigma_l)\ \omega(\epsilon\v x^{\v a})^{m}\;.$$
We now use Lemma \ref{gaussproduct} to each of the sets $K_j$ to get
$$c_0={q^{l}\over q-1}(q-1)^l\prod_{j=1}^lQ_{|K_j|}(q)$$
and for $m\not\is0\mod{\q}$,
$$c_m=(q-1)^{l-1}g(a_1m)\cdots g(a_km)\;.$$
Thus we get
$$|{\cal V}_{\lambda}(\fqstar)|=
{1\over q-1}\prod_{j=1}^lQ_j(q)+{1\over q^l(q-1)}\sum_{m=1}^{q-2}g(a_1m)\cdots
g(a_km)\omega(\epsilon\lambda)^m\;,$$
as asserted.
\end{proof}

\end{document}